\documentclass[12pt]{elsarticle}
\usepackage{amsmath,amssymb,amsthm}
\usepackage{color,graphicx,hyperref}

\makeatletter
\def\ps@pprintTitle{%
 \let\@oddhead\@empty
 \let\@evenhead\@empty
 \def\@oddfoot{\centerline{\thepage}}%
 \let\@evenfoot\@oddfoot}
\makeatother

\theoremstyle{plain}

\newtheorem{theorem}{Theorem}

\newtheorem{observation}[theorem]{Observation}
\newtheorem{corollary}[theorem]{Corollary}

\newtheorem{conjecture}[theorem]{Conjecture}

\theoremstyle{definition}

\theoremstyle{remark}

\renewcommand{\epsilon}{\varepsilon}

\newcommand{\F}{\mathcal{F}}

\newcommand{\R}{\mathbb{R}}

\begin{document}

\title{Piercing Numbers in Approval Voting}


\author[su]{Francis Edward Su\corref{correspondingauthor}}
\author[zerbib]{Shira Zerbib}
\address[su]{Department of Mathematics, Harvey Mudd College, Claremont CA 91711. E-mail: su@math.hmc.edu}
\address[zerbib]{Department of Mathematics
University of Michigan, Ann Arbor MI 48109.  E-mail: zerbib@umich.edu}
\cortext[correspondingauthor]{Corresponding author.}


\begin{abstract}
We survey a host of results from discrete geometry that have bearing on the analysis of geometric models of approval voting.  Such models view the political spectrum as a geometric space, with geometric constraints on voter preferences.  
Results on piercing numbers then have a natural interpretation in voting theory, and we survey their implications for various classes of geometric constraints on voter approval sets.
\end{abstract}

\maketitle
\section{Introduction}
The theory of set intersections has a natural connection to the study of approval voting.  In approval voting, voters cast votes for as many candidates as they wish, and the candidate with the most votes wins the election \cite{brams-fishburn83}.  Approval voting has been championed as an election system that tends to elect moderate candidates and reduces the incentive to vote strategically \cite{brams-fishburn, laslier, myerson-weber93}, has desirable game-theoretic properties \cite{laslier-sanver10}, and can be a favorable alternative to plurality voting (e.g., see \cite{bouton-castanheira, bouton-castanheira-saguer, weber, myerson02}).

The connection between approval voting and geometry arises from thinking of the political spectrum as a geometric space.  By \emph{political spectrum}, we refer to the set of all possible political positions that voters can hold.  Presently, for simplicity, we shall assume that every point of the political spectrum is represented by a candidate, so that `candidate' and `position on the spectrum' are synonymous.  The political spectrum is often modeled as a line, with conservative candidates on the right and liberal candidates on the left.  In other settings, a spectrum could be multi-dimensional \cite{mazur}, or circular \cite{hardin}, etc.  

Berg et\ al.\ \cite{berg} initiated the following model that ties intersections of geometric sets to approval voting.
We assume each voter has an \emph{approval set}: the positions on the spectrum that she finds acceptable to vote for.
Such a set may have some restrictions that are natural for a given problem.  For instance, in $\R^{d}$, an approval set is naturally a convex set if when the voter approves candidates $x$ and $y$, she would also approve any candidate on a straight line between $x$ and $y$.  But there are many other potential restrictions.

\begin{figure}
 \includegraphics[width=4in]{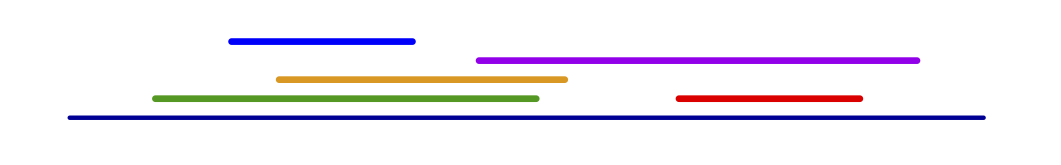}
 \caption{Approval sets of 5 voters in a linear society.}
\end{figure}

A political spectrum, with a collection of voters and their approval sets, is called a \emph{society}.  Thus we may think of a society $S$ with $n$ voters as a pair $(X, E)$ where $X$ is a political spectrum and $E$ is a collection of $n$ approval sets that are subsets of $X$.  We define the \emph{agreement proportion} $a(S)$ of a society $S$ to be the largest fraction of voters who can agree on a candidate.  If $x \in X$ is a candidate who lies in the largest number of approval sets and that number is $k$, then we call $x$ an \emph{approval winner} (there may be many) and we see that $a(S)=k/n$.  
Thus studying approval voting is equivalent to understanding the ways that sets in $X$ can intersect. 

In discrete geometry, there are a number of theorems of the following type: given some `local' intersection property, some `global' intersection property must hold.   A `local' property could be information like `every two elements of $E$ intersect' or `when we pick 10 elements in $E$ some 3 of them have a common point' and the `global' conclusion is of the form `there is a point $x$ that lies in at least half the sets of $E$' or `there is a set of 7 points where each set of $E$ contains at least one of these points'.   Such theorems about set intersections can be translated to statements about approval voting. 

For instance, it follows from a classical theorem of Helly (discussed later) that a collection of pairwise-intersecting intervals must have a point common to all the sets.  The implication for a society whose political spectrum is a line and approval sets are intervals (such a society is called a \emph{linear} society) is that if approval sets are pairwise intersecting, then there is a candidate who receives the approval of every voter, i.e., $a(S)=1$.   Here we see a rather strong hypothesis yields a rather strong conclusion, but we may be more interested in situations when we can guarantee a weaker conclusion.  For instance, what hypotheses would guarantee the approval winner gets at least half the votes, or some other fraction of the votes?  A first result in this direction is the following theorem:

\begin{theorem}[Berg et~al. \cite{berg}]
\label{agreeable-thm}
Let $S$ be a linear society such that for every subset of $m$ voters, some $k$ of them agree on a candidate. 
Then $a(S) \geq \frac{k-1}{m-1}$.
\end{theorem}

The authors note this bound can be improved slightly based on the relationship between $k$ and $m$.
If $m-1=(k-1)q + r$ for $r<k-1$, they show 
$a(S) \geq \lceil \frac{n-r}{q} \rceil / n$, where $n$ is the number of voters.  For any $n \geq m$ this gives the bound above, and as $n \rightarrow \infty$, this bound converges to $\frac{k-1}{m-1-r}$.

Since then there have been a number of other results of this type (e.g. \cite{klawe, hardin, mazur, orrison}) that consider other hypotheses on approval sets and other geometric spaces for the political spectrum $X$.  These results are all in some sense variants of Helly's Theorem.

This paper grew out of the realization that another concept from discrete geometry, the piercing number of a collection of sets, has a natural interpretation in voting theory, and a host of results about the piercing number readily translate into results about approval voting.  A {\em piercing set} of a collection of sets $E$ is a set of points intersecting every set in $E$. The \emph{piercing number} of  $E$ is the minimal size of a piercing set.  If $E$ is a collection of voter approval sets of society $S$, then a piercing set is a \emph{representative candidate set} of $S$: a set of candidates such that each voter is happy with at least one of the candidates. Thus results about piercing numbers have implications in approval voting on the size of possible representative candidate sets, as well as on the agreement proportions of societies.  

The goal of our paper is to survey a selection of these results for the social choice community who may not yet be familiar with these results from discrete geometry.  We are not attempting to be exhaustive; rather our intent is to give the reader a flavor of these results and their interpretations.  We will cite known piercing results as theorems with references, and we state their implications for approval voting as corollaries.

We should mention that there is work on the social choice side that considers approval sets and their intersections (e.g., \cite{laslier-nunez-pimienta}, \cite{nunez-xefteris} ) though these generally have not focused, as mathematicians have, on the geometric nature of the sets involved.

\section{Hypergraphs}

A society $S=(X,E)$ is an example of a structure known as a hypergraph.  A \emph{hypergraph} $H$ is a pair $(X,E)$ where $X$ is a set and $E$ is a collection of subsets of $X$.  An element of $X$ is called a {\em vertex} of $H$ and an element of $E$ is called an {\em edge} of $H$ (and thus $X$ and $E$ are the {\em vertex set} and {\em edge set} of $H$, respectively). The {\em degree} of a vertex $v$ in $X$ is the number of edges in $E$ containing $v$. The maximum degree in $H$ is denoted by $\Delta(H)$.  The terminology comes from graph theory.  A \emph{graph} is a collection of vertices and edges, where each  edge is a subset of two vertices, and a \emph{hypergraph} is a generalization of a graph in which edges can be arbitrary subsets of vertices.

In our situation, a society $S=(X,E)$ is a hypergraph in which the political spectrum $X$ is the vertex set, and the collection of all approval sets $E$ is the edge set.  Note that the term `edge' is just a reference to the hypergraph structure, and is unrelated to the geometry of approval sets that we examine in various settings.  Each `edge' in the hypergraph represents the approval set of a voter.  Then $|E|$ is the number of voters, and the degree $\Delta(S)$ is the size of the largest number of approval sets that contain a point of $X$ in common, which we will call the \emph{agreement number}.  Then the agreement proportion, the maximal fraction of sets that contain a point in common, is given by
$$a(S):=\frac{\Delta(S)}{|E|}.$$



Given a hypergraph $H=(X,E)$, a \emph{cover}  of $H$ is a set $C\subseteq X$ such that every edge of $H$ contains a point in $C$, namely, for every $e\in E$ we have $e\cap C\neq \emptyset$.  A cover gets its name from the fact that it touches every edge in $E$, so notice that a cover is precisely a piercing set of $E$.  And when the hypergraph is a society $S=(X,E)$, a cover is just a representative candidate set of $S$.

We can define the piercing number $\tau(H)$ of a hypergraph $H=(X,E)$ to be the piercing number of its edge set $E$.  It is sometimes also called the {\em covering} number or {\em stabbing} number.  The following is an easy observation on the connection between the agreement proportion and piercing number of a hypergraph $H$:
\begin{observation}\label{obs}
$a(H) \ge \frac{1}{\tau(H)}$. 
\end{observation}
\begin{proof}
Let $H=(X,E)$. Suppose that $C\subseteq X$ is a piercing set of $H$ of size $\tau(H)$. Then every edge in $E$ intersects $C$. By the pigeonhole principle there exists a point $c\in C$ that intersects at least $\frac{|E|}{\tau(H)}$ edges, and thus the degree of $c$ is at least $\frac{|E|}{\tau(H)}$. Thus $\Delta(H)\ge \frac{|E|}{\tau(H)}$, implying the observation.    
\end{proof}

Thus one way to give lower bounds on $a(S)$ is to bound from above the piercing number $\tau(S)$. This is where discrete geometry comes to our aid; in many cases, upper bounds on piercing numbers have been extensively studied.  In the next sections, we look at several classes of piercing number results and their implications.

\section{$\R^d$-convex societies}

Helly's theorem, proved by Edouard Helly in 1913 and later published in \cite{helly}, asserts that if a family $F$ of convex sets in $\R^d$ has the property that every $d+1$ sets in $F$ have a non-empty intersection, then all the sets in $F$ have a non-empty intersection. Equivalently, we can state Helly's theorem as follows:

\begin{theorem}[Helly's theorem]\label{helly}
Suppose that $H$ is a hypergraph on vertex set $\R^d$ and edge set $E$ consisting of convex sets in $X$. If every $d+1$ sets in $E$ intersect, then $\tau(H)=1$.   
\end{theorem}

\begin{figure}
 \includegraphics[width=2.5in]{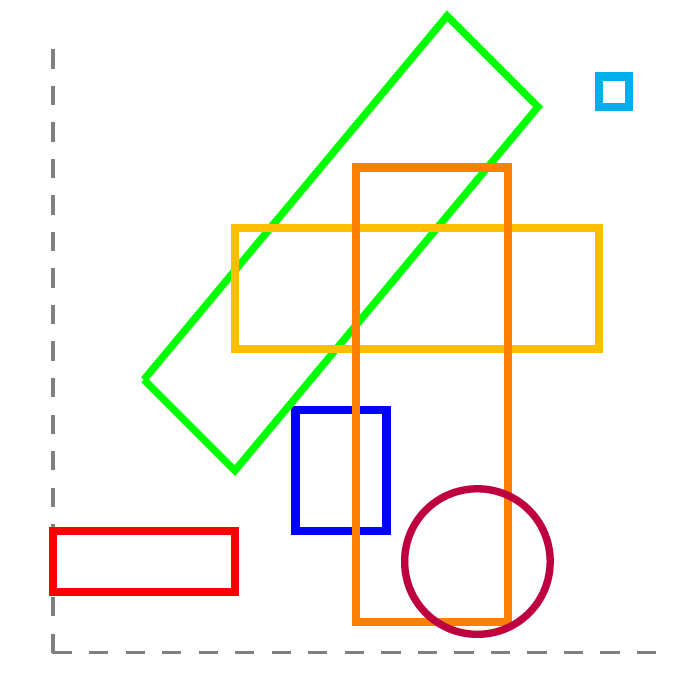}
 \caption{Approval sets of 7 voters in an $\R^2$-convex society.}
\end{figure}

Following \cite{berg}, we define an {\em $\R^{d}$-convex society} to be a society $S=(\R^d, E)$ in which $E$ consists of convex sets.  
Observation \ref{obs} then implies the following corollary:
\begin{corollary}[Berg et~al.~\cite{berg}]\label{hellycor}
Suppose that $S$ is an $\R^{d}$-convex society.  If every $d+1$ voters agree on a candidate, then $a(S)=1$, that is, there exists a candidate representing all the voters. 
\end{corollary}

Helly's theorem initiated the broad area of research in discrete geometry that deals with questions about piercing numbers in families of convex (or ``almost convex'') sets satisfying certain intersection properties.   

Given integers $p\ge q >1$, a family $\F$ of sets is said to satisfy the {\em $(p,q)$-property} if among every $p$ sets in $\F$ there exist $q$ sets with a non-empty intersection.  


In this terminology, Helly's theorem says that if a family $\F$ of convex sets in $\R^d$ satisfies the $(d+1,d+1)$-property then $\tau(\F)=1$. Finding the piercing number of families of sets in $\R^d$ satisfying the $(p,q)$-property has been known in the literature as the {\em$(p,q)$-problem}. 

Hadwiger and Debrunner \cite{HD} conjectured in 1957 that the $(p,q)$-property in a family $\F$ of convex sets in $\R^d$ implies that $\tau(\F)$ is bounded by a constant depending on $d$, $p$, and $q$. They proved this under the condition that $(d-1)p < d(q-1)$ in the stronger following form:
\begin{theorem}[Hadwiger--Debrunner  \cite{HD}]\label{HD} 
Let $\F$ be a finite family of convex sets in $\R^d$ satisfying the $(p,q)$-property for $p\ge q>1$. If $(d-1)p < d(q-1)$ then $\tau(\F)\le p-q+1$.  
\end{theorem}

In 1992 Alon and Kleitman \cite{AK} resolved the Hadwiger--Debrunner conjecture, proving that in families of convex sets in $\R^d$ that satisfy the $(p,q)$-property, the piercing numbers are bounded by a constant:   

\begin{theorem}[Alon--Kleitman \cite{AK}]\label{pqtheorem}
Let $p\ge q\ge d+1$ be integers. Then there exists a constant $c=c(d;p,q)$ depending only on $d,p,q$, such that if a family $\F$ of convex sets in $\R^d$ satisfies the $(p,q)$-property then $\tau(\F)\le c$. 
\end{theorem}

In the language of societies, the $(p,q)$-property for approval sets becomes the condition that among every $p$ voters, some $q$ of them can approve a common candidate.  Berg et~al.~\cite{berg} call such a society \emph {$(q,p)$-agreeable}, since ``some $q$ out of every $p$ voters can agree on a candidate''.  It is unfortunate that the parameters $p,q$ are in a different order for the $(p,q)$-property and $(q,p)$-agreeability, but it is easy to remember which one is which because $q$ is always at most $p$.

By Observation \ref{obs}, we have:
\begin{corollary}\label{pqcor}
Let $p\ge q\ge d+1$ be integers. Then there exists a constant $t=t(d;p,q)$ depending only on $d,p,q$, such that in any $\R^d$-convex $(q,p)$-agreeable society, we have $a(S)\ge t$. 
\end{corollary}

In general, the bounds given by Alon and Kleitman's proof for $c(d;p,q)$ (and hence for $t(d;p,q)$) are far from being optimal. For example, the Alon--Kleitman proof gives $c(2;4,3)\le 253$, however, in \cite{432} it was proved that at most $13$ points are needed to pierce a family of convex sets in $\R^2$ that satisfies the $(4,3)$-property. We can therefore conclude:
\begin{corollary}
If $S$ is an $\R^2$-convex $(3,4)$-agreeable society, then there exist $13$ candidates representing all voters.
\end{corollary}
Observation \ref{obs} then implies that $a(S)\ge \frac{1}{13}$ in this instance. However, a better bound was proved in \cite{berg} which implies in this instance that 
$a(S)\geq 0.09$.  This bound follows from the following more general result based on the Fractional Helly Theorem:
\begin{theorem} [Berg et~al.~\cite{berg}]
Let $d\ge 1$ and $p\ge q\ge 2$ be integers.  Then in every 
 $\R^d$-convex $(q,p)$-agreeable society,
 $$ a(S) \geq 1-\left(1- \frac{\binom q{d+1}}{\binom p{d+1}}\right)^{\frac{1}{d+1}}.$$
\end{theorem}


Since there is no known example of a family of convex sets in $\R^2$ that satisfies the $(4,3)$-property for which the piercing number exceeds 3, we make the following conjecture:   
\begin{conjecture}
If $S$ is an $\R^2$-convex $(3,4)$-agreeable society, then there exist $3$ candidates representing all voters. In particular, $a(S)\ge \frac{1}{3}$.   
\end{conjecture}

Over the last few decades extensive research has been done to improve the Alon--Kleitman bounds, see e.g., \cite{432, KT, KST}. For an excellent survey on the $(p,q)$-problem we refer the reader to \cite{Eckhoff}. 

Practically speaking, some may wonder how useful such theorems are, since it may be hard for a society to satisfy $(q,p)$-agreeability. However, even when such theorems fail, they may fail softly, i.e., if a society isn't $(q,p)$-agreeable, large subset may be, so that a corresponding result can be obtained. 

\section{ $\R^d$-ball societies}
In many cases the upper bounds on the piercing number improve significantly if we restrict ourselves to families of ``nice" sets.   
One such example is a result by Danzer, who proved:

\begin{theorem}[Danzer \cite{Danzer}]
If a family of disks in $\R^2$ satisfies the $(2,2)$-property, then $\tau(\F)\le 4$. 
\end{theorem}

A society $S=(\R^d,E)$ where $E$ consists of balls in $\R^d$ will be called an \emph{$\R^d$-ball} society. Danzer's theorem implies:
\begin{corollary}
In an $\R^2$-ball society $S$, 
if every two voters agree on a candidate then there are $4$ candidates representing all the voters. In particular, $a(S)\ge \frac{1}{4}$. 
\end{corollary}

A generalization of Danzer's theorem was proved by Karasev \cite{karasev}, which implies:
\begin{corollary}
Let $S$ be an $\R^d$-ball society in which every $d$ voters agree on some candidate. 
\begin{itemize}
\item If $d=3,4$ then there exist $4(d+1)$ candidates representing all the voters. In particular, $a(S)\ge \frac{1}{4(d+1)}$. 
\item If $d>5$ then there exist $3(d+1)$ candidates representing all the voters. In particular, $a(S)\ge \frac{1}{3(d+1)}$. 
\end{itemize}
\end{corollary}

Kyn\v{c}l and Tancer \cite{KT} showed that if we restrict ourselves to families $\F$ of {\em unit disks} in $\R^2$, namely, disks with radius 1 each, or to families of line segments in $\R^d$, then the $(4,3)$-property implies $\tau(F)\le 3$.  
We call the corresponding conditions for societies an \emph{$\R^2$-unit disk society} and 
\emph{$\R^d$-line segment society}, respectively. We have:
\begin{theorem}
Let $S$ be an $\R^2$-unit disk society or an $\R^d$-line segment society.  If in every set of 4 voters some 3 agree on a candidate, then there exist 3 candidates representing all voters. In particular, $a(S)\ge \frac{1}{3}$.   
\end{theorem}  

\begin{figure}[h]
 \includegraphics[width=3in]{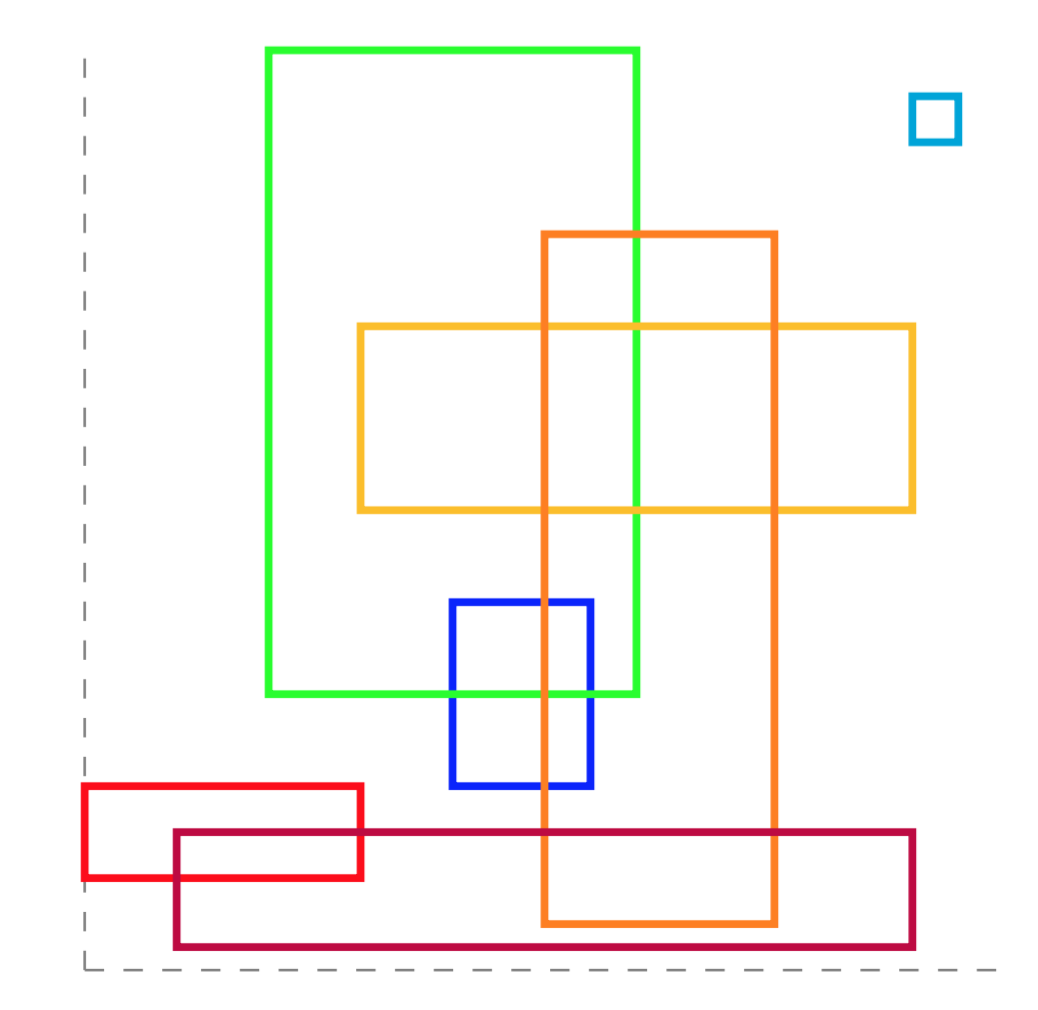}
 \caption{Approval sets of 7 voters in an $\R^2$-box society.}
\end{figure}

\section{ $\R^d$-box societies}

Another example of ``nice" sets are axis-parallel boxes. A society $S$ on vertex set $\R^d$ in which the approval sets are axis-parallel boxes will be called an {\em $\R^d$-box} society. 
Axis-parallel boxes can arise naturally as approval sets in cases where the political spectrum is multidimensional, and the axes represent separate issues. Each voter specifies an interval on each axis, representing their approval or disapproval of that single issue, and a candidate is approved by the voter if and only if the voter approved the candidate with respect to every issue. This guarantees that the preferences over a product is a box.


\begin{theorem}[Berg et al. \cite{berg}]
Let $d\geq 1$ and $p,q\geq 2$ be integers with $q\le p\le 2q-2$, and
let $S$ be a $(q,p)$-agreeable $\R^d$-box society with $n$ voters. Then
$$ a(S) \geq 1-\frac{p-q}{n},$$ 
and this bound is best possible.
\end{theorem}

Thus when $q$ is more than about half the size of $p$, the agreement proportion is rather large and close to 1 for a large number of voters. When $q=p\ge 2$ (that is, for pairwise intersecting boxes) this result implies that there is a candidate that belongs to every voter approval set, and thus in this case, $a(S)=1$.  

We do not have strong general results when $q < p/2 + 1$.  Many have examined the case $q = 2$.
In particular, the following is a long-standing conjecture in dimension $2$:
\begin{conjecture}[Wegner \cite{wegner}, 
G\'{y}arf\'{a}s--Lehel \cite{gyarfaslehel}]\label{wegner}
If a family $\F$ of axis-parallel rectangles in $\R^2$ satisfies the $(p,2)$-property, then 
$\tau(\F)\le 2(p-1)$.  
\end{conjecture}
Conjecture \ref{wegner}, if true, would imply that $(2,p)$-agreeable $\R^2$-box societies $S$ satisfy $a(S)\ge \frac{1}{2(p-1)}$. 

K{\'a}rolyi \cite{karolyi} proved that in families $\F$ of axis-parallel boxes in $\R^d$ that satisfy the $(p,2)$-property we have $\tau(\F) \leq (p-1) \left(1 + \log\left(p-1\right)\right)^{d-1}$. We thus have:
\begin{corollary}\label{boxes}
If $S$ is a $(2,p)$-agreeable $\R^d$-box society, then 
$$a(S) \ge \frac{1}{(p-1) \left(1 + \log\left(p-1\right)\right)^{d-1}}.$$  
\end{corollary}

\section{Societies with Interval and $d$-interval approval sets}

For political spectra that are lines, closed intervals are natural approval sets to consider, since they represent voters with the property that when the voter likes two candidates they would also like any candidate in between.  As indicated above, these are called \emph{linear} societies \cite{berg} and Theorem \ref{agreeable-thm} was the first result on linear societies.  Of course, linear societies are also $1$-box societies, but the results of the prior section are not stronger than Theorem \ref{agreeable-thm} when $d=1$.

Recall that Helly's theorem for $d=1$ implies that if $\F$ is a family of intervals in which every two intervals intersect, then all the intervals in $\F$ intersect.  
This fact was generalized by Gallai for families of intervals that satisfy the $(p,2)$-property:
\begin{theorem}[Gallai, (see \cite{gyarfaslehel})]\label{gallai}
If $\F$ is a finite family of intervals in $\R$ satisfying the $(p,2)$-property, then $\tau(\F)\le p-1$. 
\end{theorem}

Gallai's theorem implies for $(2,p)$-agreeable linear societies that $a(S)\ge \frac{1}{p-1}$, a result also implied by Theorem \ref{agreeable-thm}.  However, Theorem \ref{agreeable-thm} does not imply Gallai's theorem.

Gallai's theorem can be further generalized to families of 
 intervals that satisfy the $(p,q)$-property. 
 Indeed, by Theorem \ref{HD}, if $p\ge q>d$ are integers satisfying the condition $(d-1)p < d(q-1)$, then 
every family $\F$ of convex sets in $\R^d$ that satisfy the $(p,q)$-property has $\tau(\F)\le p-q+1$.  
Note that in the case $d=1$ the above condition always holds. Therefore we have:

\begin{corollary}
In a $(q,p)$-agreeable linear society where $p\geq q \geq 2$ there exist $p-q+1$ candidates representing all voters.
\end{corollary}

This implies that such societies satisfy $a(S) \geq \frac{1}{p-q+1}$; 
however, Theorem \ref{agreeable-thm} implies for such societies that 
$$a(S)\ge \frac{q-1}{p-1}$$
which is the same conclusion when $q=2$ and is stronger when $q>3$.

What if the approval sets are not intervals but intervals with gaps?  Such situations can occur if voters generally approve an interval of positions, but want to rule out intermediate candidate(s) for reasons unrelated to their political position.  For instance, a voter likes candidates $x$ and $z$, but dislikes some candidate $y$ for personal reasons.
Such situations give rise to the approval sets that are $d$-intervals.

\begin{figure}
 \includegraphics[width=3.5in]{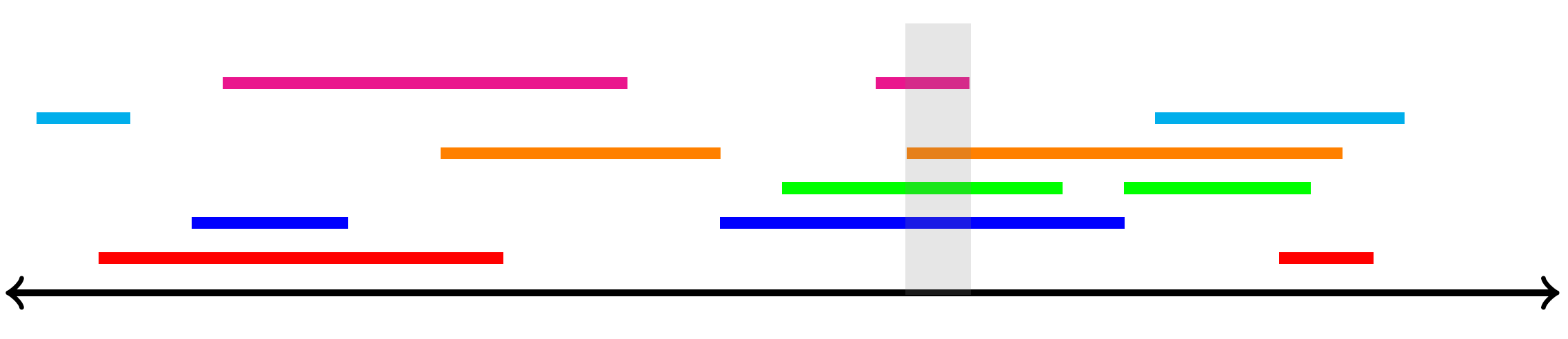}
 \caption{Approval sets of 6 voters in a $2$-interval society with  approval proportion $4/6$.}
\end{figure}

A {\em $d$-interval} is the union of at most $d$ disjoint closed intervals in $\R$.  
A society $S=(\R,E)$ will be called a {\em $d$-interval society} and each approval set is a $d$-interval.  
The case of $2$-interval societies (\emph{double interval societies}) was studied in \cite{klawe}.

It turns out that finding the piercing numbers of families of $d$-intervals is not an easy task. 
Tardos \cite{tardos} and Kaiser \cite{kaiser}
used topological methods to prove that 
if a family of $d$-intervals satisfies the $(p,2)$-property then
$\tau(H) \leq (d^2-d+1)(p-1)$. 
This translates as follows:
\begin{corollary}
Let $S$ be a $(2,p)$-agreeable $d$-interval society. Then there exist $d^2-d+1$ candidates representing all voters. In particular, $$a(S)\ge \frac{1}{d^2-d+1}.$$ 
\end{corollary}
When $d=2$ this answers the question of \cite{klawe}, who conjectured that $a(S) \geq \frac{1}{3}$.

Matou\v{s}ek \cite{matousek} showed that this bound is not far from the truth: there are examples of families of
$d$-intervals with the $(2,2)$-property in which $\tau = c(\frac{d^2}{\log d})$ for some constant $c$.

The notion of a $d$-interval is closely related to the notion of a \emph{separated $d$-interval}, which is defined to be the union of $d$ (possibly empty) intervals, one on each of $d$ fixed parallel lines.

This corresponds to a society which voters express preferences over $d$ different issues, 
and each issue has a political spectrum that is a line.  
Furthermore each voter has an approval set over each issue that is a closed interval.  
A voter will be happy if they can get one of their approved candidates elected in at least one of these spectra.

We call such a society a \emph{separated $d$-interval society}.  Such a society $S = (X,E)$ is one in which $X$ is the union of $d$ copies of $\R$ and each element of $E$ is a separated $d$-interval.
Tardos \cite{tardos} and Kaiser \cite{kaiser} studied these as well and determined for 
a family of separated $d$-intervals satisfying the $(p,2)$-property,
$\tau(H) \leq (d^2-d)(p-1)$. 

Kaiser and Rabinovich \cite{KR} proved that if a family of separated $d$-intervals satisfies the $(p,p)$-property for $p=\lceil \log_2(d + 2) \rceil$, then $\tau(H) = d$. Thus we have:
\begin{corollary}
In a separated $d$-interval society $S$, if every $\lceil \log_2(d + 2) \rceil$ voters agree on a candidate, then there exist $d$ candidates representing all voters. In particular, 
$$a(S)\ge \frac{1}{d}.$$
\end{corollary}

Zerbib \cite{Zpq} proved a generalization of these bounds to families of $d$-intervals that satisfy the $(p,q)$-property. She proved that 
if $H$ is a family of $d$-intervals that satisfies the $(p,p)$-property for some integer $p>1$, then
$\tau(H) \le p^{\frac{1}{p-1}}d^{\frac{p}{p-1}} + d$, 
and if $H$ satisfies the 
$(p,q)$-property for some integers $p\ge q>1$, then $$\tau(H)\le \max\Big\{\frac{2(ep)^\frac{q}{q-1}}{q} d^{\frac{q}{q-1}} + d,~ 2p^2d\Big\}.$$
This implies the following: 
\begin{corollary}
\label{zerbib}
Let $S$ be a $d$-interval society.
\begin{itemize}
\item If $S$ satisfies the $(p,p)$-property (namely, If every $p$ voters in $S$ agree on a candidate), then there exist $p^{\frac{1}{p-1}}d^{\frac{p}{p-1}} + d$ candidates representing all voters. In particular, $$a(S)\ge \frac{1}{p^{\frac{1}{p-1}}d^{\frac{p}{p-1}} + d}.$$ 
\item If $S$ satisfies the $(p,q)$-property (namely, 
in every set of $p$ voters in $S$ there exist $q$ that agree on a candidate), then there exist  $$\max\Big\{\frac{2(ep)^\frac{q}{q-1}}{q} d^{\frac{q}{q-1}} + d,~ 2p^2d\Big\}$$ candidates representing all voters. In particular, $$a(S)\ge \min\Big\{\frac{q}{2(ep)^\frac{q}{q-1}d^{\frac{q}{q-1}}+dq},~ \frac{1}{2p^2d}\Big\}.$$ 
\end{itemize}
\end{corollary}

\section{$d$-tree Societies}

Viewed as a discrete object, a family of $d$-intervals is a collection of subgraphs of a path $G$, each consisting of at most $d$ connected components. Alon \cite{alon1} extended this setting to collections of subgraphs of a tree $G$. Let $G$ be a tree and let $E$ be a collection of subgraphs of $G$, each consisting of at most $d$ connected components.

The corresponding social choice model has been studied \cite{niedermaier}.  
Consider a society $S=(X,H)$  whose spectrum $X$ is a vertex set of a tree $G$, and each approval set in $H$ is an induced subgraph of $G$ that has at most $d$ connected components.  As discussed in \cite{niedermaier}, $G$ may be a tree of train tracks and each approval set corresponds to a different train company and records on which portion of the track they have trains running.  
A piercing set then corresponds to good locations for stations where riders can transfer between various train companies.

Alon proved that if $H$ is a family of $d$-trees that satisfy the $(p,2)$-property, then $\tau(H)\le 2(p-1)d^2$. This implies:
\begin{corollary}
If $S$ is a $(2,p)$-agreeable $d$-tree society, then there exist $2(p-1)d^2$ candidates representing all voters. In particular, 
$$a(S)\ge \frac{1}{2(p-1)d^2}.$$
\end{corollary}

Zerbib \cite{Zpq} showed that the bounds in Corollary \ref{zerbib} extend to $d$-trees as well. Hence:
\begin{corollary}
The bounds in Corollary \ref{zerbib} extend to $d$-trees societies.  
\end{corollary}

\section{Final remarks}

1. There is some literature on \emph{circular societies} (namely societies whose approval sets are circular arcs) \cite{hardin, niedermaier}. Hardin's result \cite{hardin} is that in a circular $(q,p)$-agreeable society $S$ we have $a(S) \ge \frac{q-1}{p}$, and this bound is tight. As for piercing numbers of such families, it is easy to see that an upper bound on the piercing numbers of a families of circular arcs with the $(p,q)$-property is the upper bound on the piercing numbers of families of intervals satisfying the same property, plus 1. Indeed, given a family of circular arcs, choosing one point on the circle arbitrarily and removing all members in the family that are pierced by this point, we obtain a family of intervals.  

\begin{figure}
 \includegraphics[width=2.5in]{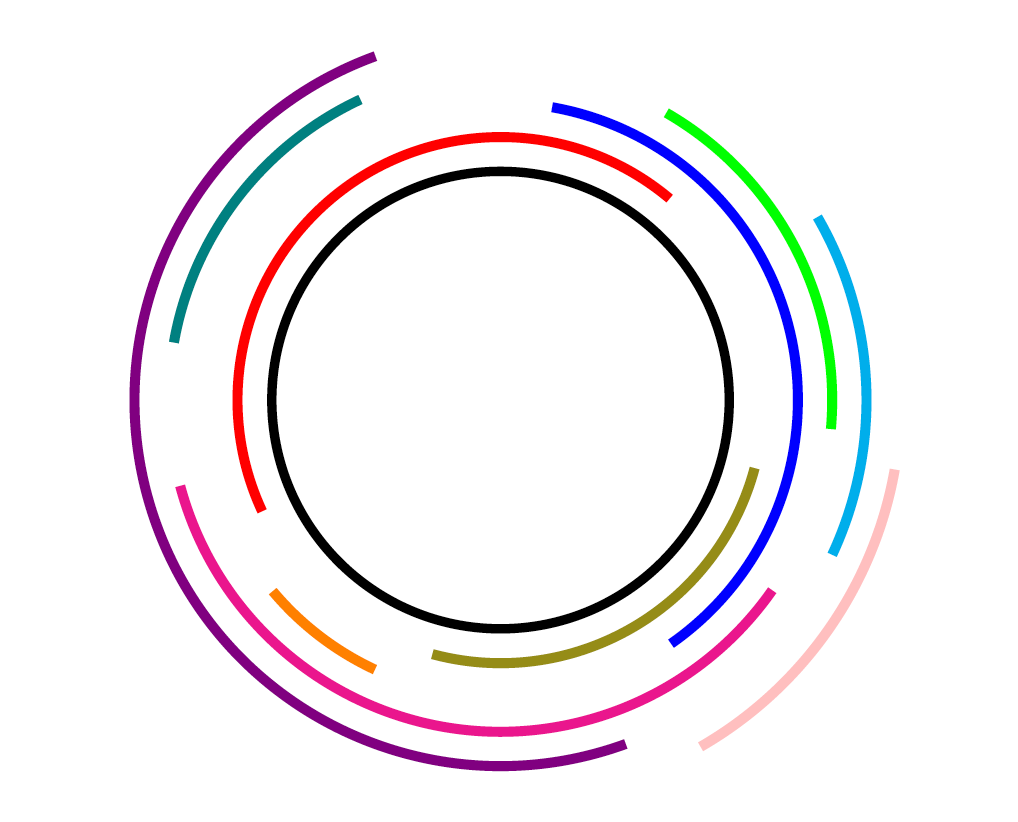}
 \caption{Approval sets of 10 voters in a circular society.}
\end{figure}

2. Fractional and weighted variants of the piercing number are known in the literature. A {\em fractional cover} of a hypergraph $H=(X,E)$ is a function $f:X\to \R_{\ge 0}$ such that $\sum_{v\in e}f(v) \ge 1$ for every $e\in E$. The {\em fractional piercing number}, usually denoted by $\tau^*(H)$, is $\min \sum_{v\in X} f(v)$, where the minimum is taken over all fractional covers $f$ of $H$. 
Can a corresponding {\em fractional approval proportion} be defined? Does this number carry some practical meaning in the theory of approval voting?

Similarly, given a hypergraph $H=(X,E)$ with an integral weight system $w(e)\ge 1$ for every edge $e\in E$, one can define a {\em $w$-cover} as a multiset $M$ of vertices in $X$ for which $|M\cap e|\ge w(e)$ for every $e\in E$. The {\em $w$-piercing number} $\tau_w(H)$ is $\min |M|$, where the minimum is taken over all $w$-covers $M$ of $H$. What is the corresponding weighted approval proportion? What practical meaning does it have? 
\medskip   

3. Our last remark is about discrete models in approval voting. In this survey we assumed (as was done previously in this area) that candidates exist at every point of a political spectrum.  However, in practice, this may not be the case.  This means that when approval sets intersect, there may not necessarily be a candidate in the intersection (though for linear societies this cannot happen, see \cite{berg}). For societies that are not linear, what can be said about approval proportions when there are a finite number of candidates in a political spectrum? 

Let us call a model in which the number of candidates is considered to be finite a {\em discrete model}. The {\em rank} of a hypergraph $H=(X,E)$ is the maximal size of an edge in $E$. Translated to the language of approval voting, the {\em rank} of a society with a discrete model is the maximal size of an approval set. The following is a simple observation on hypergraphs of rank $r$: 
\begin{observation}
If a hypergraph $H$ of rank $r$ satisfies the $(p,2)$-property then $\tau(H)\le r(p-1)$. 
\end{observation}
\begin{proof}
Let $E' \subset E$ be a subset of pairwise disjoint edges in $E$. Then $|E'| \le p-1$ and every edge in $E$ intersects some edge in $E'$. Thus the set $\bigcup E'$ is a cover of $H$, and its size is at most $r(p-1)$.  
\end{proof}
Thus we have:
\begin{corollary}
In a society $S$ with a discrete $(2,p)$-agreeable model of rank $r$, $a(S)\ge \frac{1}{r(m-1)}$. 
\end{corollary}

In some cases this can be slightly improved. Conjecturally, one such case is when the discrete model is {\em $r$-partite}. An $r$-partite hypergraph $H=(X,E)$ is a hypergraph whose vertex set $X$ has a partition $X=X_1\cup\dots\cup X_r$, and every edge $e\in E$ has $|e\cap X_i|=1$ for every $1\le i\le r$. Thus a 2-partite hypergraph is just a bipartite graph. A famous conjecture of Ryser \cite{ryser} is that an $r$-partite hypergraph $H$ which satisfies the $(p,2)$-property has $\tau(H) \le (r-1)(p-1)$.
This conjecture is wide open for $r\ge 4$. The $r=2$ case is known as K\"onig's theorem, and the $r=3$ case was settled by Aharoni \cite{aharoni}. Let us define a discrete $r$-partite model accordingly. 
Then the above translates to approval voting:
\begin{conjecture}
Suppose that a society $S$ with a discrete $r$-partite model has the property that in every subset of $p$ voters some $2$ of them agree on a candidate (that is, it satisfies the $(p,2)$-property). Then $a(S)\ge \frac{1}{(r-1)(p-1)}$. 
\end{conjecture}
\begin{theorem}
Let $S$ be a $(2,p)$-agreeable society. If $S$ has a discrete bipartite model then $a(S)\ge \frac{1}{p-1}$, and if $S$ has a discrete 3-partite  model then $a(S)\ge \frac{1}{2(m-1)}$. 
\end{theorem}  

What about piercing numbers in discrete hypergraphs (namely hypergraphs with finitely many vertices) with no bounded rank?  
Doignon \cite{Doignon} proved  a Helly-type theorem for families in which every set is obtained by the intersection of a convex set $C\subset \R^d$ with the integer lattice $\mathbb{Z}^d$; it states that a finite family of convex sets in $\R^d$ intersects at a point of $\mathbb{Z}^d$ if every $2^d$ of members of the family intersect at a point of $\mathbb{Z}^d$.
Thus we have:
\begin{theorem}
Suppose that a society $S$ has a discrete model in which the candidates are represented by the integer lattice $\mathbb{Z}^d$, and each approval set is given by an intersection of a convex set $C\subset \R^d$ with $\mathbb{Z}^d$. If every $2^d$ agree on a candidate then all the voters agree on a candidate, namely $a(S)=1$. 
\end{theorem}

There are related works where the integer lattice is replaced by some other proper set $L \subset \R^d$ (see e.g., 
\cite{DeLoera} and the reference therein). 
The more general $(p,q)$-property in such families has not been studied.

\section*{Acknowledgment}
This material is based upon work supported by the National Science Foundation under Grant No. DMS-1440140 while the authors were in residence at the Mathematical Sciences Research Institute in Berkeley, California, during the Fall 2017 semester.

\bibliography{piercing-voting}

\begin{thebibliography}{10}
\expandafter\ifx\csname url\endcsname\relax
  \def\url#1{\texttt{#1}}\fi
\expandafter\ifx\csname urlprefix\endcsname\relax\def\urlprefix{URL }\fi
\expandafter\ifx\csname href\endcsname\relax
  \def\href#1#2{#2} \def\path#1{#1}\fi

\bibitem{brams-fishburn83}
S.~J. Brams, P.~C. Fishburn, Approval voting, Birkh\"auser, Boston, Mass.,
  1983.

\bibitem{brams-fishburn}
S.~J. Brams, P.~C. Fishburn, Approval voting, American Political Science Review
  72~(3) (1978) 831--847.

\bibitem{laslier}
J.-F. Laslier, The leader rule: A model of strategic approval voting in a large
  electorate, Journal of Theoretical Politics 21~(1) (2009) 113--136.

\bibitem{myerson-weber93}
R.~B. Myerson, R.~J. Weber, A theory of voting equilibria, American Political
  Science Review 87~(1) (1993) 102--114.

\bibitem{laslier-sanver10}
J.-F. Laslier, M.~R. Sanver,
  \href{https://doi.org/10.1007/978-3-642-02839-7_8}{The basic approval voting
  game}, in: Handbook on approval voting, Stud. Choice Welf., Springer,
  Heidelberg, 2010, pp. 153--163.
\newblock \href {http://dx.doi.org/10.1007/978-3-642-02839-7_8}
  {\path{doi:10.1007/978-3-642-02839-7_8}}.
\newline\urlprefix\url{https://doi.org/10.1007/978-3-642-02839-7_8}

\bibitem{bouton-castanheira}
L.~Bouton, M.~Castanheira, \href{https://doi.org/10.3982/ECTA9111}{One person,
  many votes: divided majority and information aggregation}, Econometrica
  80~(1) (2012) 43--87.
\newblock \href {http://dx.doi.org/10.3982/ECTA9111}
  {\path{doi:10.3982/ECTA9111}}.
\newline\urlprefix\url{https://doi.org/10.3982/ECTA9111}

\bibitem{bouton-castanheira-saguer}
L.~Bouton, M.~Castanheira, A.~Llorente-Saguer, Divided majority and information
  aggregation: Theory and experiment, Journal of Public Economics 134 (2016)
  114--128.

\bibitem{weber}
R.~J. Weber, Approval voting, Journal of Economic Perspectives 9~(1) (1995)
  39--49.

\bibitem{myerson02}
R.~B. Myerson, \href{https://doi.org/10.1006/jeth.2001.2830}{Comparison of
  scoring rules in {P}oisson voting games}, J. Econom. Theory 103~(1) (2002)
  219--251, political science.
\newblock \href {http://dx.doi.org/10.1006/jeth.2001.2830}
  {\path{doi:10.1006/jeth.2001.2830}}.
\newline\urlprefix\url{https://doi.org/10.1006/jeth.2001.2830}

\bibitem{mazur}
K.~Mazur, M.~Sondjaja, M.~Wright, C.~Yarnall,
  \href{https://doi.org/10.1080/00029890.2018.1390370}{Approval voting in
  product societies}, Amer. Math. Monthly 125~(1) (2018) 29--43.
\newblock \href {http://dx.doi.org/10.1080/00029890.2018.1390370}
  {\path{doi:10.1080/00029890.2018.1390370}}.
\newline\urlprefix\url{https://doi.org/10.1080/00029890.2018.1390370}

\bibitem{hardin}
C.~S. Hardin, \href{https://doi.org/10.4169/000298910X474970}{Agreement in
  circular societies}, Amer. Math. Monthly 117~(1) (2010) 40--49.
\newblock \href {http://dx.doi.org/10.4169/000298910X474970}
  {\path{doi:10.4169/000298910X474970}}.
\newline\urlprefix\url{https://doi.org/10.4169/000298910X474970}

\bibitem{berg}
D.~E. Berg, S.~Norine, F.~E. Su, R.~Thomas, P.~Wollan,
  \href{https://doi.org/10.4169/000298910X474961}{Voting in agreeable
  societies}, Amer. Math. Monthly 117~(1) (2010) 27--39.
\newblock \href {http://dx.doi.org/10.4169/000298910X474961}
  {\path{doi:10.4169/000298910X474961}}.
\newline\urlprefix\url{https://doi.org/10.4169/000298910X474961}

\bibitem{klawe}
M.~M. Klawe, K.~L. Nyman, J.~N. Scott, F.~E. Su,
  \href{https://doi.org/10.1090/conm/624/12471}{Double-interval societies}, in:
  The mathematics of decisions, elections, and games, Vol. 624 of Contemp.
  Math., Amer. Math. Soc., Providence, RI, 2014, pp. 135--146.
\newblock \href {http://dx.doi.org/10.1090/conm/624/12471}
  {\path{doi:10.1090/conm/624/12471}}.
\newline\urlprefix\url{https://doi.org/10.1090/conm/624/12471}

\bibitem{orrison}
M.~Davis, M.~E. Orrison, F.~E. Su,
  \href{https://doi.org/10.1090/conm/624/12477}{Voting for committees in
  agreeable societies}, in: The mathematics of decisions, elections, and games,
  Vol. 624 of Contemp. Math., Amer. Math. Soc., Providence, RI, 2014, pp.
  147--157.
\newblock \href {http://dx.doi.org/10.1090/conm/624/12477}
  {\path{doi:10.1090/conm/624/12477}}.
\newline\urlprefix\url{https://doi.org/10.1090/conm/624/12477}

\bibitem{laslier-nunez-pimienta}
J.-F. Laslier, M.~a. N\'u\~nez, C.~Pimienta,
  \href{https://doi.org/10.1016/j.geb.2017.04.002}{Reaching consensus through
  approval bargaining}, Games Econom. Behav. 104 (2017) 241--251.
\newblock \href {http://dx.doi.org/10.1016/j.geb.2017.04.002}
  {\path{doi:10.1016/j.geb.2017.04.002}}.
\newline\urlprefix\url{https://doi.org/10.1016/j.geb.2017.04.002}

\bibitem{nunez-xefteris}
M.~N\'u\~nez, D.~Xefteris,
  \href{https://doi.org/10.1016/j.jet.2017.05.003}{Implementation via approval
  mechanisms}, J. Econom. Theory 170 (2017) 169--181.
\newblock \href {http://dx.doi.org/10.1016/j.jet.2017.05.003}
  {\path{doi:10.1016/j.jet.2017.05.003}}.
\newline\urlprefix\url{https://doi.org/10.1016/j.jet.2017.05.003}

\bibitem{helly}
E.~Helly, {\"U}ber mengen konvexer k{\"o}rper mit gemeinschaftlichen punkte.,
  Jahresbericht der Deutschen Mathematiker-Vereinigung 32 (1923) 175--176.

\bibitem{HD}
H.~Hadwiger, H.~Debrunner, \href{https://doi.org/10.1007/BF01898794}{\"uber
  eine {V}ariante zum {H}ellyschen {S}atz}, Arch. Math. (Basel) 8 (1957)
  309--313.
\newblock \href {http://dx.doi.org/10.1007/BF01898794}
  {\path{doi:10.1007/BF01898794}}.
\newline\urlprefix\url{https://doi.org/10.1007/BF01898794}

\bibitem{AK}
N.~Alon, D.~J. Kleitman,
  \href{https://doi.org/10.1016/0001-8708(92)90052-M}{Piercing convex sets and
  the {H}adwiger-{D}ebrunner {$(p,q)$}-problem}, Adv. Math. 96~(1) (1992)
  103--112.
\newblock \href {http://dx.doi.org/10.1016/0001-8708(92)90052-M}
  {\path{doi:10.1016/0001-8708(92)90052-M}}.
\newline\urlprefix\url{https://doi.org/10.1016/0001-8708(92)90052-M}

\bibitem{432}
D.~J. Kleitman, A.~Gy\'arf\'as, G.~T\'oth,
  \href{https://doi.org/10.1007/s004930100020}{Convex sets in the plane with
  three of every four meeting}, Combinatorica 21~(2) (2001) 221--232, paul
  Erd\H os and his mathematics (Budapest, 1999).
\newblock \href {http://dx.doi.org/10.1007/s004930100020}
  {\path{doi:10.1007/s004930100020}}.
\newline\urlprefix\url{https://doi.org/10.1007/s004930100020}

\bibitem{KT}
J.~Kyn\v{c}l, M.~Tancer,
  \href{http://www.combinatorics.org/Volume_15/Abstracts/v15i1r27.html}{The
  maximum piercing number for some classes of convex sets with the
  {$(4,3)$}-property}, Electron. J. Combin. 15~(1) (2008) Research Paper 27,
  16.
\newline\urlprefix\url{http://www.combinatorics.org/Volume_15/Abstracts/v15i1r27.html}

\bibitem{KST}
C.~Keller, S.~Smorodinsky, G.~Tardos,
  \href{https://doi.org/10.1137/1.9781611974782.148}{On max-clique for
  intersection graphs of sets and the {H}adwiger-{D}ebrunner numbers}, in:
  Proceedings of the {T}wenty-{E}ighth {A}nnual {ACM}-{SIAM} {S}ymposium on
  {D}iscrete {A}lgorithms, SIAM, Philadelphia, PA, 2017, pp. 2254--2263.
\newblock \href {http://dx.doi.org/10.1137/1.9781611974782.148}
  {\path{doi:10.1137/1.9781611974782.148}}.
\newline\urlprefix\url{https://doi.org/10.1137/1.9781611974782.148}

\bibitem{Eckhoff}
J.~Eckhoff, \href{https://doi.org/10.1007/978-3-642-55566-4_16}{A survey of the
  {H}adwiger-{D}ebrunner {$(p,q)$}-problem}, in: Discrete and computational
  geometry, Vol.~25 of Algorithms Combin., Springer, Berlin, 2003, pp.
  347--377.
\newblock \href {http://dx.doi.org/10.1007/978-3-642-55566-4_16}
  {\path{doi:10.1007/978-3-642-55566-4_16}}.
\newline\urlprefix\url{https://doi.org/10.1007/978-3-642-55566-4_16}

\bibitem{Danzer}
L.~Danzer, Zur {L}\"osung des {G}allaischen {P}roblems \"uber {K}reisscheiben
  in der {E}uklidischen {E}bene, Studia Sci. Math. Hungar. 21~(1-2) (1986)
  111--134.

\bibitem{karasev}
R.~N. Karasev, \href{https://doi.org/10.1007/s00454-007-9040-z}{Piercing
  families of convex sets with the {$d$}-intersection property in {$\Bbb
  R^d$}}, Discrete Comput. Geom. 39~(4) (2008) 766--777.
\newblock \href {http://dx.doi.org/10.1007/s00454-007-9040-z}
  {\path{doi:10.1007/s00454-007-9040-z}}.
\newline\urlprefix\url{https://doi.org/10.1007/s00454-007-9040-z}

\bibitem{wegner}
G.~Wegner, \href{https://doi.org/10.1007/BF03008396}{\"uber eine
  kombinatorisch-geometrische {F}rage von {H}adwiger und {D}ebrunner}, Israel
  J. Math. 3 (1965) 187--198.
\newblock \href {http://dx.doi.org/10.1007/BF03008396}
  {\path{doi:10.1007/BF03008396}}.
\newline\urlprefix\url{https://doi.org/10.1007/BF03008396}

\bibitem{gyarfaslehel}
A.~Gy\'arf\'as, J.~Lehel,
  \href{https://doi.org/10.1016/0012-365X(85)90045-7}{Covering and coloring
  problems for relatives of intervals}, Discrete Math. 55~(2) (1985) 167--180.
\newblock \href {http://dx.doi.org/10.1016/0012-365X(85)90045-7}
  {\path{doi:10.1016/0012-365X(85)90045-7}}.
\newline\urlprefix\url{https://doi.org/10.1016/0012-365X(85)90045-7}

\bibitem{karolyi}
G.~K\'arolyi, \href{https://doi.org/10.1007/BF02280661}{On point covers of
  parallel rectangles}, Period. Math. Hungar. 23~(2) (1991) 105--107.
\newblock \href {http://dx.doi.org/10.1007/BF02280661}
  {\path{doi:10.1007/BF02280661}}.
\newline\urlprefix\url{https://doi.org/10.1007/BF02280661}

\bibitem{tardos}
G.~Tardos, \href{https://doi.org/10.1007/BF01294464}{Transversals of
  {$2$}-intervals, a topological approach}, Combinatorica 15~(1) (1995)
  123--134.
\newblock \href {http://dx.doi.org/10.1007/BF01294464}
  {\path{doi:10.1007/BF01294464}}.
\newline\urlprefix\url{https://doi.org/10.1007/BF01294464}

\bibitem{kaiser}
T.~Kaiser, \href{https://doi.org/10.1007/PL00009315}{Transversals of
  {$d$}-intervals}, Discrete Comput. Geom. 18~(2) (1997) 195--203.
\newblock \href {http://dx.doi.org/10.1007/PL00009315}
  {\path{doi:10.1007/PL00009315}}.
\newline\urlprefix\url{https://doi.org/10.1007/PL00009315}

\bibitem{matousek}
J.~Matou\v~sek, \href{https://doi.org/10.1007/s00454-001-0037-8}{Lower bounds
  on the transversal numbers of {$d$}-intervals}, Discrete Comput. Geom. 26~(3)
  (2001) 283--287.
\newblock \href {http://dx.doi.org/10.1007/s00454-001-0037-8}
  {\path{doi:10.1007/s00454-001-0037-8}}.
\newline\urlprefix\url{https://doi.org/10.1007/s00454-001-0037-8}

\bibitem{KR}
T.~Kaiser, Y.~Rabinovich,
  \href{https://doi.org/10.1007/PL00009421}{Intersection properties of families
  of convex {$(n,d)$}-bodies}, Discrete Comput. Geom. 21~(2) (1999) 275--287.
\newblock \href {http://dx.doi.org/10.1007/PL00009421}
  {\path{doi:10.1007/PL00009421}}.
\newline\urlprefix\url{https://doi.org/10.1007/PL00009421}

\bibitem{Zpq}
S.~Zerbib, The $(p, q) $ property in families of $ d $-intervals and $ d
  $-trees, arXiv preprint arXiv:1703.02939.

\bibitem{alon1}
N.~Alon, \href{https://doi.org/10.1016/S0012-365X(02)00427-2}{Covering a
  hypergraph of subgraphs}, Discrete Math. 257~(2-3) (2002) 249--254, kleitman
  and combinatorics: a celebration (Cambridge, MA, 1999).
\newblock \href {http://dx.doi.org/10.1016/S0012-365X(02)00427-2}
  {\path{doi:10.1016/S0012-365X(02)00427-2}}.
\newline\urlprefix\url{https://doi.org/10.1016/S0012-365X(02)00427-2}

\bibitem{niedermaier}
A.~Niedermaier, D.~Rizzolo, F.~E. Su,
  \href{https://doi.org/10.1090/conm/625/12492}{A tree {S}perner lemma}, in:
  Discrete geometry and algebraic combinatorics, Vol. 625 of Contemp. Math.,
  Amer. Math. Soc., Providence, RI, 2014, pp. 77--92.
\newblock \href {http://dx.doi.org/10.1090/conm/625/12492}
  {\path{doi:10.1090/conm/625/12492}}.
\newline\urlprefix\url{https://doi.org/10.1090/conm/625/12492}

\bibitem{ryser}
H.~J. Ryser, Neuere probleme der kombinatorik, Vortr{\"a}ge {\"u}ber
  Kombinatorik, Oberwolfach (1967) 69--91.

\bibitem{aharoni}
R.~Aharoni, \href{https://doi.org/10.1007/s004930170001}{Ryser's conjecture for
  tripartite 3-graphs}, Combinatorica 21~(1) (2001) 1--4.
\newblock \href {http://dx.doi.org/10.1007/s004930170001}
  {\path{doi:10.1007/s004930170001}}.
\newline\urlprefix\url{https://doi.org/10.1007/s004930170001}

\bibitem{Doignon}
J.-P. Doignon, \href{https://doi.org/10.1007/BF01949705}{Convexity in
  cristallographical lattices}, J. Geometry 3 (1973) 71--85.
\newblock \href {http://dx.doi.org/10.1007/BF01949705}
  {\path{doi:10.1007/BF01949705}}.
\newline\urlprefix\url{https://doi.org/10.1007/BF01949705}

\bibitem{DeLoera}
J.~A. De~Loera, R.~N. La~Haye, D.~Oliveros, E.~Rold\'an-Pensado,
  \href{https://doi.org/10.1515/advgeom-2017-0028}{Helly numbers of algebraic
  subsets of {$\Bbb R^d$} and an extension of {D}oignon's theorem}, Adv. Geom.
  17~(4) (2017) 473--482.
\newblock \href {http://dx.doi.org/10.1515/advgeom-2017-0028}
  {\path{doi:10.1515/advgeom-2017-0028}}.
\newline\urlprefix\url{https://doi.org/10.1515/advgeom-2017-0028}

\end{thebibliography}

\bibliographystyle{elsarticle-num}

\end{document}